\documentclass[a4paper,11pt]{amsart}
\usepackage{amsfonts,amssymb,epsfig,latexsym}
\usepackage{amsmath}
\usepackage[latin1]{inputenc}
\usepackage{color,layout}
\usepackage{ifthen}
\usepackage{graphicx}

\usepackage{pdfsync}
\usepackage{graphicx}
\usepackage{color}

\newtheorem{theorem}{Theorem}[section]
\newtheorem{lemma}[theorem]{Lemma}
\newtheorem{proposition}[theorem]{Proposition}

\newtheorem{remark}[theorem]{Remark}

\def\square{\hbox{\vrule\vbox{\hrule\phantom{o}\hrule}\vrule}}

\def\R{\mathbb {R}}
\def\C{\mathbb {C}}

\def\re{\mathop{\rm Re}\nolimits}
\def\im{\mathop{\rm Im}\nolimits}
\def\la{\langle}
\def\ra{\rangle}
\def\O{\mathcal O}

\def\eps{\varepsilon}
\newcommand{\indic}{1\!\!1}

\newcommand{\be}{\begin{equation}}
\newcommand{\ee}{\end{equation}}

\numberwithin{equation}{section}

\begin{document}

\title[Excited resonance widths for Helmholtz resonators]{Excited resonance widths for Helmholtz resonators with straight neck}

\begin{abstract} 
We consider resonances associated with excited eigenvalues of the cavity of a general Helmholtz resonator with straight neck. Under the assumption that the neck stays away from the nodal set of the corresponding eigenstate, we generalise the optimal exponential lower bound on the width of the resonance, that we have obtained in a previous paper for the ground resonance only.
\end{abstract}

\author[Duyckaerts, Grigis, Martinez]{Thomas Duyckaerts${}^{1,2}$, Alain Grigis${}^1$ \& Andr\'e Martinez${}^3$}


\maketitle

\addtocounter{footnote}{1}
\footnotetext{Université Sorbonne Paris Cité, Universit\'e Paris 13, Institut Galil\'ee, D\'epartement de Math\'ematiques, avenue J.-B. Cl\'ement, 93430 Villetaneuse, France}
\addtocounter{footnote}{1}
\footnotetext{Institut Universitaire de France}
\addtocounter{footnote}{1}
\footnotetext{Universit\`a di Bologna, Dipartimento di
Matematica, Piazza di Porta San Donato 5, 40127 Bologna,
Italy.}

\setcounter{section}{0}

\section{Introduction}

We continue our study of the Helmholtz resonator with straight neck, started in \cite{DGM}. Let us recall that a resonator consists of a bounded cavity (the chamber) connected to the exterior by a thin tube (the neck). The sounds produced by it are mathematically described by the resonances of the Dirichlet Laplacian $-\Delta_{\Omega}$ on the domain $\Omega$ consisting of the union of the chamber, the neck, and the exterior.

Many works have been devoted to the location of such resonances, and, in particular, to estimates on their widths (that is, the absolute value of their imaginary part): see, e.g., \cite{Be, Fe, FL, HM, MN1, MN2} and references therein. 

In our previous paper \cite{DGM}, we studied the width of the first resonance $\rho (\varepsilon)$ of a general Helmholtz resonator with a straight neck of size $\varepsilon <<1$ and length $L$. Using several Carleman estimates, including estimates with limiting Carleman weights (see \cite{KSU,DKSU}), we proved  that
\be
\label{exactlim}
\lim_{\varepsilon \to 0_+} \varepsilon \ln |\im\rho (\varepsilon)| = -2\alpha_0 L,
\ee
where $\alpha_0$ is the ground-state energy of the cross section of the neck.

Here, we generalise this result to higher resonances, that is to resonances close to (non degenerate) excited states of the cavity (see Theorem \ref{mainth}). 

From a physical point of view, the real part of the resonances  correspond to the frequencies of the sound produced, while the inverse of their width represents their duration (or life-time). 
To illustrate this claim, we give the long-time asymptotic behaviour of the solutions to the wave equation in the resonator that are spectrally concentrated near such energies (see Theorem \ref{propevol}). In particular, this behaviour highlights the fact that the frequencies of the sounds produced by the resonator are determined by the shape of the chamber, while their duration by the length and the width of the neck. 

The general procedure used in \cite{DGM} consisted in first relating the width of the resonance to the size of the resonant state $u$ in the exterior domain, then to relate this size to that of the corresponding Dirichlet eigenfunction $v$ of the domain ``cavity+neck'' inside the neck, and finally to that of the corresponding Dirichlet eigenfunction of the cavity in a neighbourhood of the points of contact between the neck and the cavity.

It is important to observe that all this procedure remains valid for excited resonances, except the last argument that strongly relies on the fact that the function $v$ keeps a constant sign in the neck. However, in the case where this last property is satisfied, then all the proof of \cite{DGM} remains valid and leads to \eqref{exactlim}. Therefore, our main effort in the current paper will consist in proving this non-vanishing property of $v$ in (and close to) the neck.

For this purpose, it is natural to add the assumption that the points of contact between the tube and the cavity stay away from the nodal set of the corresponding Dirichlet eigenfunction of the cavity alone. Then, in order to deduce the non vanishing of $v$ near the neck, it would be sufficient to have some stability results on the nodal set under perturbation of the domain. As far as we know, such results are not available in the literature, and it seems that the main difficulty comes from the fact that, for general Dirichlet eigenfunctions, the number of nodal domains may change after perturbation, without precise control on it. In the last section, we prove that, actually, this number cannot increase (see Proposition \ref{stabnod}), but this is still not enough to deduce the result on the resonator.

To overcome this difficulty, we have to compare more precisely $v$ with the eigenfunction of the cavity, and this is done by comparing the various resolvents that are involved, through the abstract separation of the cavity and the neck (in a spirit similar to that of \cite{HM}). In that way, it can be proved that $v$ cannot vanish at the points of the cavity that are close to the neck. Then, the final key-argument consists in the observation that the volume of a nodal domain cannot tend to 0 as $\varepsilon \to 0$ (this may be seen as a consequence of the Faber-Krahn inequality, but for our purpose we only need a weaker property that we prove in an elementary way: see Section \ref{section5}), so that, in particular, the neck of the resonator cannot contain any nodal domain.

Concerning the long time behaviour of the solutions to the wave equation, in addition to the main expected exponential term, we also provide the asymptotic of the remainder term as $\varepsilon \to 0$, and our result on the widths permits us to give precise estimates on the time interval where the exponential term constitutes the main contribution (see Theorem \ref{propevol} and Remark \ref{timeinterval}). 
Our result should be compared to that of \cite{TZ} or \cite[Theorem 2]{BuZw} that treat the case of some fixed compact perturbation of the Euclidean space, and obtain asymptotic expansions of solutions to the wave equation in term of resonances of the corresponding operator.

As a by-product of our proof, we also obtain a generalisation to excited states of the result of \cite{BHM1, BHM2} concerning the splitting of eigenvalues in a system consisting of two symmetric cavities connected by a thin tube (see Remark \ref{splitt}).

The plan of the paper is as follows: In the next section, we give the precise assumptions and state our result. The third section contains backgrounds results, and a reduction of the problem. Section 4 is devoted to precise estimates on the difference between the resonant state and the corresponding Dirichlet eigenfunction of the cavity. The proof of the main result is completed in Section 5, while the long time behaviour for the wave equation is obtained in Section 6. Finally, Section 7 contains a result on the number of nodal domains under perturbation.

{\bf Acknowledgments}. We would like to thank an anonymous referee of \cite{DGM} for his suggestion to relate our results on resonances to the long-time behaviour of the solutions to the  wave equation for the Helmholtz resonator. A.M. would like to thank the LAGA of the University Paris 13 for its hospitality during January and June 2017, that has made possible this collaboration.

\section{Assumptions and results}
\label{S:results}

Let $\mathcal C$ and $\mathcal B$ be two bounded open sets in $\R^n$ ($n\geq 2$), with $\mathcal C^\infty$ boundary, and denote by $\overline{\mathcal C}$, $\overline{\mathcal B}$
their closures, and by $\partial{\mathcal C}$, 
$\partial{\mathcal B}$ their boundaries. We assume that $\mathcal C$ is connected, $\mathcal B$ is contractible,  and that Euclidean coordinates $x=(x_1,\dots, x_n)=:(x_1,x')$ can be chosen in such a way that, for some 
$L>0$,  one has, 
\be
\label{hyp1}
\overline{\mathcal C}\subset \mathcal B\,\,; \,\, 0\in \partial{\mathcal C} \,\,; \, \,M_0:=(L,0_{\R^{n-1}})\in \partial{\mathcal B}\,\,; \,\,
{[0,L]}\times\{0_{\R^{n-1}}\}\subset \overline{\mathcal B}\backslash {\mathcal C}.
\ee
We also assume,
\be
\label{hyp2}
{[0,L]}\times\{0_{\R^{n-1}}\} \mbox{ is transversal to } \partial{\mathcal B} \mbox{ at } M_0,\mbox{ and to }\partial{\mathcal C} \mbox{ at } 0_{\R^{n}}.
\ee

Let  $D_1\subset \R^{n-1}$ be a bounded domain containing the origin, with smooth boundary $\partial D_1$. For $\varepsilon >0$ small enough, we set $D_\varepsilon:=\varepsilon D_1$ and,
$$
\begin{aligned}
&{\mathbf E}:= \R^n \backslash \overline{\mathcal B};\\
&{\mathcal T}(\varepsilon):= ([-\varepsilon_0, L+\varepsilon_0]\times D_{\varepsilon}) \cap \left(\R^n\backslash ({\mathbf E}\cup{\mathcal C})\right);\\
&\mathcal C(\varepsilon) = {\mathcal C} \cup \mathcal{T}(\varepsilon),
\end{aligned}
$$
where $\varepsilon_0>0$ is fixed sufficiently small in order that $[-\varepsilon_0, L+\varepsilon_0]\times\{0_{\R^{n-1}}\}$ crosses $\partial{\mathcal C}$ and $\partial{\mathcal B}$ at one point only.
Then,
the resonator is defined as, 
$$\Omega(\varepsilon):={\mathcal C}(\varepsilon) \cup{\mathbf E}.$$
(See Figure \ref{fig}).
\begin{figure}[h]
\caption{The resonator}
\label{fig} 
\includegraphics[scale=1]{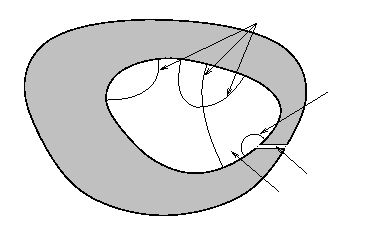}
 \put(-113,27){$\mathcal{C}(\varepsilon)=\mathcal{C}\cup \mathcal{T}(\varepsilon)$}
\put(-130,205){$G_0$ (Nodal set of $u_0$)}
 \put(-100,40){Cavity $\mathcal{C}$}
 \put(-65,53){$\mathcal{T}(\varepsilon)$}
\put(-44,145){$\{|x|=\delta\}$}
 \put(-107,88){$\scriptstyle{0}$}
 \put(-74,86){$\scriptstyle{M_0}$}
 \put(-200,50){$\mathcal{B}$}
\end{figure}
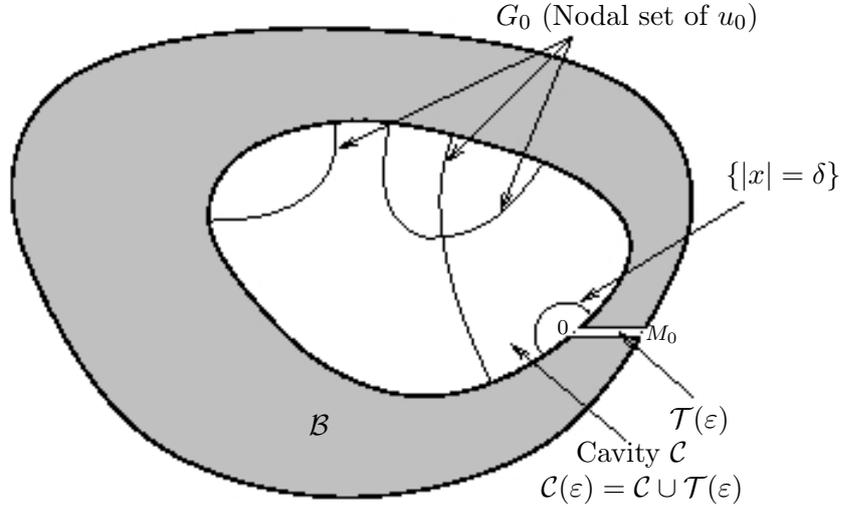

For any domain $Q$, let $P_Q$ denote the Laplacian $-\Delta_Q$ on the interior $\overset{\circ}{Q}$ of $Q$ with Dirichlet boundary conditions on $\partial Q$,
and set $P_\varepsilon:= P_{\Omega_\varepsilon}$. 

The resonances of $P_\varepsilon$ are defined as the eigenvalues of the operator obtained by performing a complex dilation with
respect to $x$, for $|x|$ large (see e.g. \cite{HM}, \cite{DGM}).

We are interested in those
resonances of $P_\varepsilon$ that are close to the eigenvalues of $P_{\mathcal C}$. Thus, let $\lambda_0 >0$ be an eigenvalue 
of $P_{\mathcal{C}}$ with $u_0$ the  corresponding normalized eigenfunction. In \cite{DGM}, we assumed that $\lambda_0$ was the lowest eigenvalue of $P_{\mathcal C}$, and under this condition, we proved that the unique resonance $\rho(\varepsilon)\in\mathbb{C}$ of  $P_\varepsilon$ such that $\rho(\varepsilon)\to \lambda_0$ as $\varepsilon\to 0_+$, satifies
\be
\lim_{\varepsilon\to 0_+} \varepsilon \ln |\im \rho(\varepsilon) | = -2\alpha_0 L,
\ee
where $\alpha_0$ is the square root of the first eigenvalue of $-\Delta_{D_1}$.

The importance of this assumption on $\lambda_0$ came from the fact that, in that case, the unique eigenvalue of $\lambda(\varepsilon)$ of 
$P_{\mathcal{C}(\varepsilon)}$ close to $\lambda_0$ is necessarily its lowest eigenvalue, too, and in particular the corresponding eigenfunction $v_\varepsilon$ keeps a constant sign inside the tube ${\mathcal T}(\varepsilon)$. This was essential in order to perform the last argument of the proof of the result: see \cite{DGM}, Section 6.

Here, we plan to remove this assumption, in order to treat excited resonances, too. 
We assume,

\be
\label{hyp3}
\lambda_0 \mbox{ is a simple eigenvalue of } P_{\mathcal C}.
\ee

Then, by the results of \cite{HM}, we know there exists a unique resonance $\rho(\varepsilon)\in\mathbb{C}$ of  $P_\varepsilon$ such that $\rho(\varepsilon)\to \lambda_0$ as $\varepsilon\to 0_+$.

We denote by $u_0$ the normalised eigenfunction of $P_{\mathcal C}$ associated with $\lambda_0$, and by $G_0$ its nodal set, that is,
\be
G_0:=\overline{\{x\in {\mathcal C}\, ;\, u_0(x)=0\}},
\ee
and we further assume,

\be
\label{hyp4}
0_{\R^{n}}\notin G_0.
\ee

We now state our main result.
\begin{theorem}\sl 
\label{mainth}
Under Assumptions (\ref{hyp1}), (\ref{hyp2}), (\ref{hyp3}), and (\ref{hyp4}), one has,
\be
\lim_{\varepsilon\to 0_+} \varepsilon \ln |\im \rho(\varepsilon) | = -2\alpha_0 L.
\ee
\end{theorem}

 \section{Background}
 
 As already observed in \cite{MN1, MN2, DGM}, thanks to \cite{HM} we only need to prove that, for any $\delta >0$, there exists $C_\delta>0$ such that, for all $\varepsilon >0$ small enough, one has,
 \be
 \label{result1}
 |\im \rho(\varepsilon)|\geq \frac1{C_\delta}e^{-2\alpha_0(1+\delta)L/\varepsilon}.
 \ee
 
We denote by $u_\varepsilon$ the resonant function associated with $\rho(\varepsilon)$, that is, the non trivial outgoing (analytic) solution of the equation $-\Delta u_\varepsilon = \rho (\varepsilon) u_\varepsilon$ in $\Omega(\varepsilon)$, such that $u_\varepsilon\left\vert_{\partial\Omega(\varepsilon)}\right. =0$.

The fact that $u_\varepsilon$ is outgoing means that, for all $\mu >0$ small enough,  the function $x\mapsto u_\varepsilon (x+i\mu x)$ (well-defined for $|x|$ large enough) is $L^2$ at infinity.

 Moreover, we normalize
 $u_\varepsilon$  by setting, 
\be
\label{normueps}
\Vert  u_\varepsilon\Vert_{L^2({\mathcal C}(\varepsilon ))} =1.
\ee

Let us observe that all the preliminary results used in \cite{DGM} (in particular those coming from \cite{MN2}) remain valid in our case. They permit to reduce the proof of (\ref{result1}) to the fact that $|u_\varepsilon|$ cannot be exponentially smaller than $e^{-\alpha_0L/\varepsilon}$ inside ${\mathbf E}$. 

More precisely, reasoning by contradiction, if we assume the existence of $\delta_0 >0$ such that, along a sequence $\varepsilon \to 0_+$, one has
\be
\label{absurd}
|\im\rho(\eps)| =\O(e^{-2(\alpha_0+\delta_0)L/\varepsilon}).
\ee 
then, for any neighborhood ${\mathcal U}$ of $M_0$ and any compact set $K\subset \R^n$, there exists $\delta_K >0$ such that,
$$
\Vert u_\varepsilon\Vert_{H^1({\mathbf E}\cap K \backslash {\mathcal U})} =\O(e^{-(\alpha_0+\delta_K)L/\varepsilon}),
$$
uniformly as $\varepsilon \to 0_+$ along the same sequence (see \cite{DGM}, Proposition 3.2, and \cite{MN2}, Proposition 6.1).

After that, by Carleman inequalities (first near $M_0$ in ${\mathbf E}$, then inside the tube, at some fixed distance of $0_{\R^n}$), it was proved in \cite{DGM} that 
for any $r_0\in (0,L)$, there exists $\delta_1>0$ such that, 
 \be
 \label{troppetit}
 \|u_{\eps}\|_{H^1([r_0,L]\times D_{\eps})}=\O(e^{-(\alpha_0r_0+\delta_1)/\eps}),
 \ee
 uniformly for $\eps >0$ small enough (see \cite{DGM}, Proposition 5.1).  
 
 It is important to note that all the arguments leading to these reullts never use the fact that $\lambda_0$ is the lowest eigenvalue of $P_{\mathcal C}$, and thus remain valid in our situation.

This is unfortunately not the case for the last arguments in \cite{DGM} (that is, those of Section 6), where the positivity of the eigenfunctions that are involved is crucial. However, since this positivity is used in a neighborhood of $0_{\R^n}$ only, here we only need to prove that the same property holds in our situation. Namely, if we denote by $v_\varepsilon$ the normalized eigenfunction of $P_{{\mathcal C}(\varepsilon)}$ associated with the eigenvalue $\lambda_\varepsilon$ close to $\lambda_0$, it is sufficient to prove the following proposition:

\begin{proposition}\sl
\label{propdiff}
There exists $\delta >0$ such that, for any $\varepsilon >0$ small enough, $v_\varepsilon$ never vanishes in ${\mathcal T}(\varepsilon) \cup [{\mathcal C(\varepsilon)}\cap \{ |x|\leq \delta\}]$.
\end{proposition}

\section{Estimates on $u_0-v_\varepsilon$}
\label{diffp}
By assumption (\ref{hyp4}), we already know that $u_0$ keeps a constant sign near $0_{\R^n}$ in $\mathcal C$. In order to prove Proposition \ref{propdiff}, we need sufficiently good estimates on the difference $u_0-v_\varepsilon$ in $\mathcal C$, in particular on its normal derivatives at the boundary $\partial\mathcal C$.

We set,
\be
Z(\eps ):= \overset{\circ}{\widehat{T(\eps)}},
\ee
and,
$$
b_\eps:= \partial{\mathcal C}\cap \overline {Z(\eps)} = \partial Z(\eps) \cap \overline {\mathcal C}.
$$
As in \cite{HM},  we set,
$$
H_D:= P_{\mathcal C}\oplus P_{Z(\eps)}
$$
acting on $L^2({\mathcal C})\oplus L^2(Z(\eps))\simeq L^2({\mathcal C}(\eps))$, with domain $(H^2\cap H^1_0)({\mathcal C})\oplus (H^2\cap H^1_0)(Z(\eps))$. We denote by $R_D(z)$ its resolvent, and by $R_\eps(z)$ the resolvent of $P_{{\mathcal C}(\eps)}$.

Let $\gamma$ be a small enough simple oriented loop in $\C$ around $\lambda_0$. Reasoning as in the proof of \cite[Proposition 4.1]{HM}, we see that, for $z\in \gamma$, we have,
\be
\label{diffres}
R_\eps (z) = R_D(z) + R_\eps (z) T_{\mathcal C}^*B_{\rm int}R_D(z),
\ee

where  $T_{\mathcal C}\, :\, H^1({\mathcal C}) \to L^2(b_\eps)$ is the trace operator on $b_\eps$ from $\mathcal C$, and 
$$
B_{\rm int}\, : \, H^2({\mathcal C})\oplus H^2(Z(\eps))\to L^2(b_\eps)
$$

is defined by,
$$
B_{\rm int}(u_1\oplus u_2) = -T_{\mathcal C}\partial_\nu u_1 + T_Z\partial_\nu u_2
$$

Here, $\partial_\nu$ stands for the outward (from ${\mathcal C}$) normal derivative on $b_\eps$, and $T_Z\, :\, H^1(Z(\eps)) \to L^2(b_\eps)$ is the trace operator on $b_\eps$ from $Z(\eps)$.

Indeed, using that $(H_D-z)R_D(z)=(P_{{\mathcal C}(\eps)}-z)R_\eps(z)=I$ on $L^2({\mathcal C}(\eps))$, and that $R_\eps(z)^* =R_\eps(\overline z)$, taking any $u,v\in L^2({\mathcal C}(\eps))$ we can write,
$$
\la (R_\eps(z)-R_D(z))u,v\ra  = \la H_DR_D(z)u, R_\eps(z)^*v\ra - \la R_D(z)u, P_{{\mathcal C}(\eps)}R_\eps(z)^*v\ra
$$

and thus, by Green formula (and after having divided the integral over ${\mathcal C}(\eps)$ into $\int_{\mathcal C} + \int_{Z(\eps)}$),
$$
\la (R_\eps(z)-R_D(z))u,v\ra  = \int_{b_\eps} \left(T_Z\partial_\nu R_D(z)u.\overline{R_\eps(z)^*v }- T_{\mathcal C}\partial_\nu R_D(z)u.\overline{R_\eps(z)^*v }\right) dS,
$$

that is,
$$
\la (R_\eps(z)-R_D(z))u,v\ra  = \la B_{\rm int} R_D(z)u, T_{\mathcal C}R_\eps(z)^*v \ra_{L^2(b_\eps)},
$$
 and (\ref{diffres}) follows.
 
As a consequence, 
\be
\Vert R_D(z)- R_\eps (z)\Vert_{L^2,L^2} \leq \Vert R_\eps(z)\Vert_{H^{-1},L^2} \Vert T_{\mathcal C}^*\Vert_{L^2, H^{-1}} \Vert B_{\rm int} R_D(z)\Vert_{L^2,L^2}
\ee

(where we have pointed-out the regularity only, since the precise spaces where the various operators act are obvious).

In particular, for $z\in \gamma$, this gives,
\be
\Vert R_D(z)- R_\eps (z)\Vert_{L^2,L^2} =\O( \Vert T_{\mathcal C}^*\Vert_{L^2, H^{-1}} \Vert B_{\rm int}R_D(z)\Vert_{L^2,L^2}),
\ee
uniformly as $\eps\to 0_+$.

Since $T_{\mathcal C}\, : \, H^1({\mathcal C}) \to L^2(b_\eps)$ is uniformly bounded (and the dual of $H^1$ is included in $H^{-1}$), so is $T_{\mathcal C}^*\, : \, L^2(b_\eps)\to H^{-1}({\mathcal C})$. Moreover, by the same arguments as in \cite{HM} (specially the proof of Lemma 4.3), we have,
\be
\label{beta}
\Vert B_{\rm int}R_D(z)\Vert_{L^2,L^2}=\O(\eps^\beta),
\ee
where $\beta =\frac12$ when $n\geq 3$, and $\beta$ is an arbitrary number in $(0,\frac12)$ when $n=2$. This finally gives,
\be
\Vert \frac1{2i\pi}\oint_\gamma (R_D(z)- R_\eps (z)) dz\Vert_{L^2({\mathcal C}(\eps))} =\O (\eps^\beta).
\ee

Therefore, possibly by renormalizing $v_\eps$, we can assume,
\be
v_\eps = \frac1{2i\pi}\oint_\gamma R_\eps (z)\widetilde u_0 dz,
\ee

where $\widetilde u_0\in L^2({\mathcal C}(\eps))$ stands for the 0-extension of $u_0$ outside $\mathcal C$. 
Of course, by construction we also have,
\be
\widetilde u_0 = \frac1{2i\pi}\oint_\gamma R_D (z)\widetilde u_0 dz
\ee

In particular, we immediately obtain,
\be
\label{u-v}
\Vert \widetilde u_0-v_\eps\Vert_{L^2({\mathcal C}(\eps))} =\O(\eps^\beta),
\ee
and, by \cite[Proposition 4.1]{HM}, we also have,
\be
\label{l0-leps}
\lambda_0-\lambda_\eps =\O(\eps^\beta).
\ee

Then, using that $-\Delta u_0=\lambda_0u_0$ and $-\Delta v_\eps =\lambda_\eps v_\eps$ in $\mathcal C$, we deduce from (\ref{u-v})-(\ref{l0-leps}),
$$
\Vert \Delta(u_0-v_\eps)\Vert_{L^2({\mathcal C})} =\O(\eps^\beta),
$$

and, by iteration, for any $m\geq 0$,
\be
\Vert \Delta^m(u_0-v_\eps)\Vert_{L^2({\mathcal C})} =\O(\eps^\beta).
\ee

Finally, we can apply  e.g. Theorem 5 at page 323 of \cite{Ev} to conclude that, for any $m\geq 0$,
\be
\Vert u_0-v_\eps\Vert_{H^m({\mathcal C})} =\O(\eps^\beta).
\ee

Therefore, by standard Sobolev inequalities and the continuity of the trace operator on $\partial\mathcal C$, we conclude,
\begin{proposition}\sl 
\label{propu-v}
One has,
$$
\sup_{\mathcal C} |u_0-v_\eps|+\sup_{\mathcal C} |\nabla u_0-	\nabla v_\eps| =\O(\eps^\beta),
$$
uniformly as $\eps\to 0_+$.
\end{proposition}
\begin{remark}\sl
This result can be found in \cite{Ev}, and it uses the smoothness of $\partial\mathcal C$. However, since we will use the estimate of Proposition \ref{propu-v} near $0_{\R^n}$ only, by standard arguments of localisation, it is possible to restrict the smoothness of $\partial\mathcal C$ near that point only. As a consequence, one could obtain a generalisation of our final result to cavities that are not necessarily smooth away from $0_{\R^n}$.
\end{remark}

\section{Positivity of the eigenfunction for small $\varepsilon$} 
\label{section5}
In this section we conclude the proof of Proposition \ref{propdiff}. 
Note that by our assumptions, $u_0$ as a constant sign in the intersection of $\mathcal{C}$ and a neighborhood of the origin. To fix ideas, we will assume that $u_0$ is positive there. By Hopf lemma (see, e.g., \cite{GiTr}), this implies,
\be
\label{hopf}
\partial_\nu u_0(0)<0.
\ee
We start by proving the following:
\begin{lemma}
\label{lem_couronne}
 \label{Lueps_positif}
Assume \eqref{hyp1}, \eqref{hyp2}, \eqref{hyp3}, \eqref{hyp4}.
Then, there exists $\delta>0$ such that, for any $\delta'\in (0,\delta)$ and for all $\eps >0$ small enough, one has,
\begin{equation}
 \label{ueps_positif}
v_{\eps}>0\, \mbox{ on } \mathcal{C}\cap\{ \delta'\leq |x|\leq \delta\}.
\end{equation} 
\end{lemma}
\begin{proof} We set $a:= -\partial_\nu u_0(0)$. By \eqref{hopf} we have $a>0$, and by the continuity of $\partial_\nu u_0$ on $\partial\mathcal C$ near 0, there exists $\delta >0$ such that $\partial_\nu u_0(x)\leq -a/2$ on $\{|x|\leq \delta\}\cap \partial\mathcal C$,
 and $u_0>0$ on $\{|x|\leq \delta\}\cap \mathcal C$. In particular, for any $c>0$, one has,
$$
\min\{ u_0(x)\, ;\, x\in {\mathcal C},\, |x|\leq \delta,\, d(x, \partial{\mathcal C})\geq c\} >0.
$$
Using Proposition \ref{propu-v}, we deduce that, for all $\eps$ small enough, we have,
$$
\begin{aligned}
&\min\{ v_\eps(x)\, ;\, x\in {\mathcal C},\, |x|\leq \delta,\, d(x, \partial{\mathcal C})\geq c\} >0;\\
&\partial_\nu v_\eps (x)\leq -a/4 \mbox{ on }\{|x|\leq \delta\}\cap \partial\mathcal C.
\end{aligned}
$$
As a consequence, and since $v_\eps = 0$ on $\partial{\mathcal C} \backslash \{ |x|\leq \eps\}$, by performing a Taylor expansion from the boundary of $\mathcal C$, we conclude that, for any $\delta'\in (0,\delta)$, there exists $c'>0$ such that,
$$
v_\eps (x) \geq c'd(x,\partial {\mathcal C}) \mbox{ on } \mathcal C\cap \{\delta'\leq |x|\leq \delta\},
$$
and the lemma is proved.
\end{proof}

We next prove Proposition \ref{propdiff}. 

We let $\delta>0$ small enough such that \eqref{ueps_positif} holds, and fix $\delta'\in (0,\delta)$. 
We also fix $\eps>0$ small enough.

We extend $v_{\eps}$ by zero to a compactly supported $H^1$ function on $\R^n$. We will still denote the extension by $v_{\eps}$. 
We consider the negative part of $v_{\eps}$, namely
$$ v_{{\eps}}^-=-\min(v_{\eps},0)=\frac{1}{2}(|v|_{\eps}-v_{\eps}).$$
By \cite[Corollary 6.18]{LiebLoss01}, $v_{\eps}^-\in H^1(\R^n)$ with 
$\nabla v_{\eps}^-(x)=-\nabla v_{\eps}(x)$ if $v_{\eps}(x)<0$, and $\nabla v_{\eps}^-(x)=0$ if $v_{\eps}(x)\geq 0$.

We set
$$ {\mathcal C}_{\delta,\eps}:=\left(\mathcal{C}\cap \{|x|<\delta\}\right)\cup \mathcal{T}(\eps),$$
and 
$$
w_{\eps}=v_{\eps}^-\indic_{ {\mathcal C}_{\delta,\eps}}.
$$

Since, by Lemma \ref{lem_couronne}, we have $v_{\eps}^-(x)=0$ if $\delta'<|x|<\delta$, we see that $w_{\eps}$ is an element of $H^1(\R^n)$ whose support is included (for small $\eps$) in $R_{\delta}$,
where $R_{\delta}$ is the cylinder
$$R_{\delta}=[-\delta,L+\eps_0]\times D_{\delta}.$$
We must prove that $w_{\eps}=0$ almost everywhere. We argue by contradiction, assuming that it is not the case. As a consequence
\begin{equation}
 \label{weps_above}
\frac{1}{\int |w_{\eps}|^2}\int |\nabla w_{\eps}|^2\geq \min_{f\in H^1_0(R_{\delta})}\frac{1}{\int |f|^2}\int |\nabla f|^2\gtrsim \frac{1}{\delta^2}+\frac{1}{L^2},
\end{equation}
where, to obtain the last bound, we have observed that the first Dirichlet eigenvalue on $R_{\delta}$ is the sum of the first Dirichlet eigenvalue on $D_{\delta}$ (of order $1/\delta^2$ by elementary scaling consideration) and the first Dirichlet eigenvalue on $[-\delta,L+\eps_0]$ (that is, $\pi/(L+\eps_0+\delta)$).

On the other hand, since $w_\eps = -v_\eps$ on Supp $w_\eps$, we have,
$$ \int |\nabla w_{\eps}|^2=-\int \nabla w_{\eps}\cdot \nabla v_{\eps}=\int w_{\eps}\Delta v_{\eps}=\lambda(\eps)\int |w_{\eps}|^2.$$
In view of \eqref{weps_above} we deduce the bound
$$ \lambda(\eps)\gtrsim \frac{1}{\delta^2}+\frac{1}{L^2},$$
which is a contradiction since $\lambda(\eps)$ converges to $\lambda_0$ as $\eps$ goes to $0$, and $\delta>0$ is any small number satisfying \eqref{ueps_positif}, thus $\delta$ can be chosen arbitrarily small by our assumptions.

\section{Long time behaviour for the wave equation} 

In this section, we investigate the asymptotic behaviour of the solution $w=w(t,x)$ to,
\be
\label{waveeq}
\begin{aligned}
& \frac{\partial^2w}{\partial t^2} - \Delta w =0\,\, \mbox{ on } {\mathcal C}(\varepsilon); \\
& w\left|_{x\in \partial{\mathcal C}(\eps)}\right. =0;\\
& w\left|_{t=0} = f_0\right. \quad ; \quad \frac{\partial w}{\partial t}\left|_{t=0} = f_1\right. ,
\end{aligned}
\ee
both as $\eps\to 0_+$ and $t\to +\infty$. Here $f_0$ is in the domain of $P_\eps$, and $f_1$ in that of $P_\eps^{1/2}$.

As it is well known, for all $t\in \R$, $w_t(x):=w(t,x)$ is abstractly given by,
$$
w_t = \cos \left (t\sqrt{P_\eps}\right)f_0 + \sin \left (t\sqrt{P_\eps}\right) P_\eps^{-1/2}f_1,
$$

that is, using Stone's formula,
\be
\label{stone}
w_t =\frac{1}{2i\pi} \int_0^{+\infty} \left(\cos \left(t\sqrt{\lambda}\right)E(\lambda)f_0 + \lambda^{-1/2}\sin \left(t\sqrt{\lambda}\right) E(\lambda)f_1\right)d\lambda
\ee
with
$$
E(\lambda):= R_\eps(\lambda+i0)-R_\eps(\lambda-i0)\quad ; \quad R_\eps(z):=(P_\eps-z)^{-1}.
$$
We assume that $f_0$ and $f_1$ are localized in energy near $\lambda_0$, in the sense that there exists $\psi\in C_0^\infty(\R;[0,1])$ such that,
\begin{itemize}
\item $\psi =1$ near $\lambda_0$;
\item ${\rm Supp}\psi \subset (0,+\infty)$;
\item $\left( {\rm Sp}(P_{\mathcal C})\backslash \{\lambda_0\}\right) \cap {\rm Supp}\psi =\emptyset$;
\item $f_j=\psi (P_\eps )f_j$ for $j=0,1$.
\end{itemize}
For simplicity, we also assume that $f_0$ and $f_1$ are compactly supported (though this could be replaced by taking them in some vector space of decaying analytic functions at infinity). Thus, by taking $\chi \in C_0^\infty(\R^n;[0,1])$ such that $\chi =1$ in a sufficiently large compact set, we can assume that $\chi f_j =f_j$ ($j=0,1$).
In that case, (\ref{stone}) becomes
$$
w_t =\frac{1}{2i\pi} \int_0^{+\infty} \psi (\lambda)\left(\cos \left(t\sqrt{\lambda}\right)E(\lambda)\chi f_0 + \lambda^{-1/2}\sin \left(t\sqrt{\lambda}\right) E(\lambda)\chi f_1\right)d\lambda
$$
and in order to be able to modify the domain of integration  into the complex (as, e.g., in \cite{BuZw} or \cite{BrMa}), we concentrate on the asymptotic behaviour of $\chi w_t$. It is well known that both $R^\chi_+(\lambda) :=\chi R_\eps(\lambda +i0)\chi$ and $R^\chi_-(\lambda):=\chi R_\eps(\lambda -i0)\chi$ extend to complex values of $\lambda$ as meromorphic functions, and the  poles of $R^\chi_{\pm}(\lambda)$ are in $\pm\im\lambda <0$. These poles are, respectively, the resonances of $P_\eps$ and their conjugates (of course, all this can be seen by using an analytic distortion away from the support of $\chi$: see, e.g. \cite{SjZw}).

Then, the asymptotic behaviour of $\chi w_t$ will result from that of the two operators,
\be
\label{AB}
\begin{aligned}
&A_\eps (t):=\chi \psi (P_\eps) \cos \left (t\sqrt{P_\eps}\right)\chi=\frac{1}{2i\pi} \int_0^{+\infty} \psi (\lambda)\cos( t\sqrt{\lambda})\left(R^\chi_+(\lambda)-R_-^\chi(\lambda)\right)d\lambda;\\
&B_\eps (t):=\chi\psi (P_\eps) \frac{\sin \left (t\sqrt{P_\eps}\right)}{\sqrt{ P_\eps}}\chi=\frac{1}{2i\pi} \int_0^{+\infty} \psi (\lambda)\frac{\sin (t\sqrt{\lambda})}{\sqrt\lambda}\left(R^\chi_+(\lambda)-R_-^\chi(\lambda)\right)d\lambda.\\
\end{aligned}
\ee

As before, we denote by  $u_\eps$ the resonant state of $P_\eps$ associated with $\rho (\eps)$, normalised by \eqref{normueps}. Then, using an analytic distortion, it is possible to define $\alpha_\eps:= \int_{\Omega(\eps)} u_\eps^2(x)dx \in \C$, and, by the results of \cite[Section 6]{HM}, we know that $\alpha_\eps = 1+\O(e^{-(\alpha_0L-\delta)/\eps})$ for any $\delta >0$. 

For $f\in L^2(\Omega (\eps))$, we set,
\be
\label{Pieps}
\Pi^\chi_\eps f := \alpha_\eps^{-1}\la  f, \chi \overline u_\eps\ra \chi u_\eps,
\ee
and we denote by ${\mathbf 1}_{\mathbf E}\, : L^2(\Omega (\eps)) \to L^2({\mathbf E})$ the operator of restriction to $\mathbf E$, and by ${\mathbf 1}_{\mathbf E}^*$ its adjoint, that is, the operator of extension by 0 outside $\mathbf E$.

We have,
\begin{theorem}\sl
\label{propevol}
Under Assumptions \eqref{hyp1}, \eqref{hyp2}, \eqref{hyp3}, \eqref{hyp4},  the operators $A_\eps(t)$ and $B_\eps (t)$ defined in (\ref{AB}) satisfy,
$$
\begin{aligned}
A_\eps (t)& = \re\left[ e^{-it\sqrt{\rho(\eps)}}\Pi^\chi_\eps\right]+  \chi {\mathbf 1}_{\mathbf E}^*\psi (P_{\mathbf E})\cos \left( t\sqrt {P_{\mathbf E}}\right){\mathbf 1}_{\mathbf E} \chi  +\O(\eps^\beta t^{-\infty}) \\
& = \re\left[ e^{-it\sqrt{\rho(\eps)}}\Pi^\chi_\eps\right]+\O(t^{-\infty}) ;\\
B_\eps (t)& =-\im\left[ \frac{e^{-it\sqrt{\rho(\eps)}}}{\sqrt{\rho(\eps)}}\Pi^\chi_\eps\right]+\chi {\mathbf 1}_{\mathbf E}^*\psi (P_{\mathbf E})\frac{\sin \left (t\sqrt{P_{\mathbf E}}\right) }{\sqrt {P_{\mathbf E}}}{\mathbf 1}_{\mathbf E}\chi + \O(\eps^\beta t^{-\infty})\\
& = -\im\left[ \frac{e^{-it\sqrt{\rho(\eps)}}}{\sqrt{\rho(\eps)}}\Pi^\chi_\eps\right]+ \O(t^{-\infty}),
\end{aligned}
$$
in ${\mathcal L}(L^2(\Omega_\eps))$, uniformly as $t\to +\infty$ and $\eps\to 0_+$. 
Here $\beta$ is as in \eqref{beta}, and, for any operator $C$, $\re C:= \frac12 (C+C^*)$  stands for the real part of $C$, while $\im C:= \frac1{2i}(C-C^*)$ stands for its imaginary part.
\end{theorem}

\begin{remark} \sl
One could also slightly modify the normalisation of the resonant state by setting,
\be
w_\eps := \alpha_\eps^{-\frac12} u_\eps,
\ee
so that $\int_{\Omega(\eps)} w_\eps^2(x)dx=1$, and $\Pi^\chi_\eps$ takes the more usual form $\Pi^\chi_\eps f := \la  f, \chi \overline w_\eps\ra \chi w_\eps$.
\end{remark}

\begin{proof} Let $\gamma_\pm \subset \{\pm \im z\geq 0\}$ be two complex paths parametrized (and oriented) by $\R$ such that,
\begin{itemize}
\item $\gamma_\pm (s) =s$ for $s\notin \{\psi = 1\}$;
\item $\re \gamma_\pm (s) =s$ for all $s\in\R$;
\item $\pm \im\gamma_\pm (\lambda_0) >0$.
\end{itemize}

Then, performing as in \cite{BuZw,BrMa} a change of complex contour of integration in the region where $\psi =1$, we can write,
$$
\begin{aligned}
\chi \psi (P_\eps)e^{-it\sqrt {P_\eps}}\chi  =\frac{1}{2i\pi}\oint_{\gamma_+} \psi(\re\lambda)&e^{-it\sqrt{\lambda}} R_+^\chi (\lambda)d\lambda\\
& -\frac{1}{2i\pi}\oint_{\gamma_-} \psi(\re\lambda)e^{-it\sqrt{\lambda}}R_-^\chi (\lambda)d\lambda,
\end{aligned}
$$
and thus,
 \be
 \label{eqQ}
\chi \psi (P_\eps)e^{-it\sqrt {P_\eps}}\chi  =\frac{-1}{2i\pi}\oint_\gamma e^{-it\sqrt{\lambda}}R_+^\chi (\lambda)d\lambda + Q,
\ee
where $\gamma$ is a simple oriented complex loop around $\lambda_0$, and $Q$ is given by,
\be
\label{T+}
Q:= \frac1{2i\pi}\int_{\gamma_-}e^{-it\sqrt{\lambda}}\psi(\re\lambda)\left( R_+^\chi (\lambda)-R_-^\chi (\lambda)\right)d\lambda.
\ee
Then, using that, for all $N$, $e^{-it\sqrt{\lambda}} =(1+t)^{-N}\left( 1+2i\sqrt{\lambda}\frac{d}{d\lambda}\right)^N e^{-it\sqrt{\lambda}} $, and that $R_+^\chi (\lambda)$ (and thus also its derivatives with respect to $\lambda$) is uniformly bounded on $\gamma_-\cap \{\re\lambda\in {\rm Supp}\psi\}$, by integrations by parts we immediately obtain,
\be
\label{estQ}
\Vert Q\Vert =\O(t^{-\infty}),
\ee
uniformly when $t\to +\infty$ and $\eps\to 0_+$. Moreover, by the same arguments as in \cite[Section 4]{HM}, we see that we can approach $R_\pm^\chi(\lambda)$ by $R_{\mathcal C}(\lambda)\oplus R_Z(\lambda)\oplus \chi R^\pm_{\mathbf E}(\lambda)\chi $ (where we have denoted by  $R_W(\lambda)=(P_W-z)^{-1}$ the resolvent of the Dirichlet Laplacian on $W$) up to $\O(\eps^\beta)$. Here,
 $\chi R^\pm_{\mathbf E}(\lambda)\chi $ stands for the meromorphic extension of $\chi R_{\mathbf E}(\lambda)\chi $ from $\{ \pm \im \lambda >0\}$, and we observe  that $R_{\mathcal C}(\lambda)$ is meromorphic in a whole neighbourhood of $\lambda_0$, while $R_Z(\lambda)$ is holomorphic near $\lambda_0$. Inserting this approximation into \eqref{T+}, and taking advantage of the fact that $P_{\mathbf E}$ has no resonances near $\lambda_0$ to modify again the contour $\gamma_-$ into $\R_+$, we obtain,
$$
\begin{aligned}
Q= \frac1{2i\pi}\int_0^{+\infty}e^{-it\sqrt{\lambda}}\psi(\lambda)\chi {\mathbf 1}_{\mathbf E}^*\left( R_{\mathbf E}(\lambda +i0)-R_{\mathbf E}(\lambda -i0)\right)&{\mathbf 1}_{\mathbf E}\chi d\lambda \\
& +\O(\eps^\beta t^{-\infty}),
\end{aligned}
$$
that is,
\be
\label{Q}
Q= \chi {\mathbf 1}_{\mathbf E}^*e^{-it\sqrt{P_{\mathbf E}}}\psi(P_{\mathbf E}){\mathbf 1}_{\mathbf E}\chi +\O(\eps^\beta t^{-\infty}).
\ee

On the other hand, since $\rho(\eps)$ is the only pole of $R_+^\chi$ surrounded by $\gamma$, we have,
$$
\frac{-1}{2i\pi}\oint_\gamma e^{-it\sqrt{\lambda}}R_+^\chi (\lambda)d\lambda =-e^{-it\sqrt {\rho(\eps)}}{\rm res}\left(R_+^\chi ,\rho(\eps)\right),
$$
where ${\rm res}\left(R_+^\chi ,\rho(\eps)\right)$ is the residue of $R_+^\chi$ at $\rho(\eps)$. Standard arguments (involving for instance the complex distorsion as in \cite{HM, SjZw}, or directly working in modified spaces adapted to resonances as in \cite{HS}) show that, in our case, we have,
\be
\label{resR}
{\rm res}\left(R_+^\chi ,\rho(\eps)\right)=-\Pi^\chi_\eps,
\ee
where $\Pi^\chi_\eps$ is defined in \eqref{Pieps}. 

Therefore, gathering \eqref{eqQ}-\eqref{resR},  we have proved,
$$
\begin{aligned}
\chi \psi (P_\eps)e^{-it\sqrt {P_\eps}}\chi  & =e^{-it\sqrt {\rho(\eps)}}\Pi^\chi_\eps + \chi {\mathbf 1}_{\mathbf E}^*e^{-it\sqrt{P_{\mathbf E}}}\psi(P_{\mathbf E}){\mathbf 1}_{\mathbf E}\chi +\O(\eps^\beta t^{-\infty})\\
& =e^{-it\sqrt {\rho(\eps)}}\Pi^\chi_\eps + \O( t^{-\infty}).
\end{aligned}
$$

In the same way, we also have,
$$
\begin{aligned}
\chi \psi (P_\eps)\frac{e^{-it\sqrt {P_\eps}}}{\sqrt {P_\eps}}\chi &=\frac{e^{-it\sqrt {\rho(\eps)}}}{\sqrt {\rho(\eps)}}\Pi^\chi_\eps+ \chi {\mathbf 1}_{\mathbf E}^*\frac{e^{-it\sqrt{P_{\mathbf E}}}}{\sqrt{P_{\mathbf E}}}\psi(P_{\mathbf E}){\mathbf 1}_{\mathbf E}\chi +\O(\eps^\beta t^{-\infty})\\
&=\frac{e^{-it\sqrt {\rho(\eps)}}}{\sqrt {\rho(\eps)}}\Pi^\chi_\eps+\O(t^{-\infty}),
\end{aligned}
$$
and the result follows by observing,
$$
\begin{aligned}
&A_\eps (t):= \re\left[\chi \psi (P_\eps)e^{-it\sqrt {P_\eps}}\chi \right] ;\\
&B_\eps (t):=-\im\left[\chi \psi (P_\eps)\frac{e^{-it\sqrt {P_\eps}}}{\sqrt {P_\eps}}\chi \right].
\end{aligned}
$$
\end{proof}

\begin{remark}\sl 
\label{timeinterval}
By Theorem \ref{mainth}, we know that $\im\sqrt{\rho(\eps)}\sim -e^{-2\alpha_0L/\eps}$. As a consequence, Theorem \ref{propevol} provides the main contribution of the wave decay up to times of order $t_\eps$ such that,
$$
t_\eps={\mathcal O}(\frac1{\varepsilon}|\im\sqrt{\rho(\eps)}|^{-1}),
$$
(that is, much beyond the life-time of the wave), but not beyond times of order $T_\eps$ such that,
$$
T_\eps \geq e^{(2\alpha_0L+\delta)/\eps},
$$
for some $\delta >0$.
\end{remark}

\begin{remark}\sl
\label{splitt}
In \cite{BHM1, BHM2}, a system of two symmetric cavities connected by a thin tube is considered. It is proved that the splitting $E_2-E_1$ between the first two Dirichlet eigenvalues satisfies,
\be
\label{estsplitt}
\lim_{\eps \to 0_+} \eps \ln (E_2-E_1) = -\alpha_0L,
\ee
where, as before, $\alpha_0$ is the ground-sate energy of the cross section of the tube, and $L$ is its length. As for \cite{DGM}, the arguments used in \cite{BHM2} remain valid for more excited states, too, except the final one that relies on the non vanishing of the eigenfunction of the system ``one cavity+tube'' inside the tube. By applying our arguments, we see that this property remains valid if the tube is connected with the cavities away from the nodal set of the corresponding eigenfunctions of the two cavities. In that case, the identity \eqref{estsplitt} can be generalised to the splitting between the pair of eigenvalues of the system that are close to some arbitrary non-degenerate eigenvalue of one cavity alone.
\end{remark}

\section{A remark on nodal domains}

An alternative strategy in order to prove Proposition \ref{propdiff} would consist in considering the nodal domains of $u_0$ and $v_\eps$, that is the connected components of the complementary of their nodal set. A well-known result of Courant (see, e.g., \cite{CoHi}) asserts that if $\lambda_0$ is the $k$-th eigenvalue of $P_{\mathcal C}$, then the total number of nodal domains of $u_0$ is at most $k$. In the particular case $k=2$, this implies that this number is exactly 2 (since the second eigenfunction, being orthogonal to the first one, must vanish somewhere). Therefore, in that case, $u_0$ and $v_\eps$ have the same number of nodal domains, and it can be directly deduced from Lemma \ref{Lueps_positif} and Proposition \ref{propu-v}  that some neighbourhood of ${\mathcal T}(\eps)$ is necessarily included in the same nodal domain of $v_\eps$.

However, this argument is no more valid when $k\geq 3$. Even if $u_0$ has exactly $k$ nodal domains, it does not seem that the same property necessarily holds for $v_\eps$. 

Indeed, in this section we will show that, under perturbation, the number of nodal domains cannot increase. More precisely, in a slightly more general setting, we prove,
\begin{proposition}\sl
\label{stabnod}
Let $\Omega\subset \R^n$ be a bounded open set, and let $u_0$ an eigenfunction of the Dirichlet Laplace operator on $\Omega$, associated with the eigenvalue $\lambda_0$, and with $m$ nodal domains. For any $k\geq 1$, let also $v_k$ be an eigenfunction of the Dirichlet Laplace operator on some bounded open set $\Omega_k\subset \R^n$, associated with the eigenvalue $\lambda_k$, such that,
\begin{itemize}
\item $\lambda_k =\O(1)$ as $k\to \infty$;
\item ${\rm Volume} \left(\Omega_k \backslash \Omega\right)\to 0$  as $k\to \infty$;
\item For any compact set $K\subset \Omega$, there exists $k_0\geq 1$ such that $K\subset \Omega_k$ for all $k\geq k_0$;
\item For any compact set $K\subset \Omega$, $\sup_{K} |v_k-u_0| \to 0$ as $k\to \infty$.
\end{itemize}
Then, for $k$ large enough, the number of nodal domains of $v_k$ in $\Omega_k$ is at most $m$.
\end{proposition}
\begin{proof}
Let $U_1, \dots, U_m$ be the nodal domains of $u_0$. For each $i=1,\dots,m$, we choose $x_i\in U_i$. Since $u(x_i)\not= 0$, the assumptions imply that $v_k(x_i)\not =0$ for $k$ large enough, and thus there exists a nodal domain $V_{k,i}$ of $v_k$ such that 
$$
x_i\in V_{k,i}.
$$
Then, for an arbitrary $y\in \Omega$ such that $u(y)\not= 0$, we have $y\in U_i$ for some $i$, and thus there exists a continuous compact path $\gamma \subset U_i$ connecting $y$ to $x_i$. As a consequence, thanks to the last assumption, for $k$ large enough $v_k$ never vanishes on $\gamma$, and thus $\gamma\subset V_{k,i}$. In particular, $y\in V_{k,i}$ for $k$ large enough. Setting,
$$
\omega_k:= \Omega\backslash \bigcup_{i=0}^m V_{k,i},
$$
and using that $\{ u_0=0\}$ has measure 0, we have proved,
$$
{\bf 1}_{\omega_k} (y) \to 0 \mbox{ almost everywhere on } \Omega \mbox{, as } k\to \infty.
$$
By the Lebesgue dominated convergence theorem, we deduce,
$$
{\rm Volume}(\omega_k) \to 0 \mbox{ as } k\to \infty.
$$
Setting also,
$$
\widetilde \omega_k:= \Omega_k\backslash \bigcup_{i=0}^m V_{k,i},
$$
we have $\widetilde \omega_k\subset \omega_k\,  \cup\,  (\Omega_k\backslash\Omega)$, and thus also,
$$
{\rm Volume}(\widetilde\omega_k) \to 0 \mbox{ as } k\to \infty.
$$
But this implies that $\widetilde\omega_k$ does not contain any additional nodal domain of $v_k$. Indeed, if such a domain (say, $\Gamma_k$) existed, its volume should tend to $0$, and $\lambda_k$ would be its first Dirichlet eigenvalue. But then, let us show that $\lambda_k$ should tend to $\infty$ as $k\to \infty$ (which is in contradiction with our first assumption): 
For $n=2$ this is a consequence of the standard Faber-Krahn inequality (see, e.g., \cite{Pl}), and for $n\geq3$ it results from a generalisation of this inequality (see, e.g. \cite{DeG, BuFr}). Alternatively, for $n\geq 3$ we can also use the Gagliardo-Nirenberg-Sobolev inequality (see, e.g., \cite{Ni}),
$$
\Vert v \Vert_{L^{\frac{2n}{n-2}}(\R^n)}\leq C_n\Vert \nabla v\Vert_{L^2(\R^n)},
$$
valid for all $v\in H^1(\R^n)$, and with $C_n>0$ depending only on $n$. By Holder inequality, we also have,
$$
\Vert v\Vert_{L^2(\Gamma_k)}\leq \Vert v \Vert_{L^{\frac{2n}{n-2}}(\Gamma_k)}[{\rm Volume}(\Gamma_k)]^{\frac{1}{n}},
$$
and thus, for $v\in C_0^\infty(\Gamma_k)$,
$$
\Vert \nabla v\Vert_{L^2(\Gamma_k)} \geq \frac1{C_n}[{\rm Volume}(\Gamma_k)]^{-\frac1{n}}\Vert v\Vert_{L^2(\Gamma_k)}.
$$
As a consequence, the first Dirichlet eigenvalue of $\Gamma_k$ is greater than \\
$C_n^{-2}[{\rm Volume}(\Gamma_k)]^{-\frac2{n}}$, and therefore it tends ot $\infty$ as $k\to \infty$.

Finally, since $\sharp \{ V_{k,i}\, ; \, 1\leq i\leq m\} \leq m$, the result is proved.
\end{proof}

\end{document}